\documentclass[reqno, 12pt]{amsart}
\pdfoutput=1
\makeatletter
\let\origsection=\section \def\section{\@ifstar{\origsection*}{\mysection}}
\def\mysection{\@startsection{section}{1}\z@{.7\linespacing\@plus\linespacing}{.5\linespacing}{\normalfont\scshape\centering\S}}
\makeatother

\usepackage{amsmath,amssymb,amsthm}
\usepackage{mathrsfs}
\usepackage{mathabx}\changenotsign
\usepackage{dsfont}

\usepackage[dvipsnames]{xcolor}
\usepackage[backref]{hyperref}
\hypersetup{
    colorlinks,
    linkcolor={red!60!black},
    citecolor={green!60!black},
    urlcolor={blue!60!black}
}

\usepackage{graphicx}

\usepackage[open,openlevel=2,atend]{bookmark}

\usepackage[abbrev,msc-links,backrefs]{amsrefs}
\usepackage{doi}

\renewcommand{\PrintDOI}[1]{\doi{#1}}

\usepackage[T1]{fontenc}
\usepackage{lmodern}
\usepackage[babel]{microtype}
\usepackage[english]{babel}

\usepackage{geometry}
\numberwithin{equation}{section}
\usepackage{enumitem}

\usepackage{pgf,tikz}

\usepackage{caption}
\usepackage{subcaption}
\usepackage[labelformat=simple,labelfont={}]{subcaption}

\makeatletter
\def\greek#1{\expandafter\@greek\csname c@#1\endcsname}
\def\Greek#1{\expandafter\@Greek\csname c@#1\endcsname}
\def\@greek#1{\ifcase#1
	\or $\alpha$%
	\or $\beta$%
	\or $\gamma$%
	\or $\delta$%
	\or $\epsilon$%
	\or $\zeta$%
	\or $\eta$%
	\or $\theta$%
	\or $\iota$%
	\or $\kappa$%
	\or $\lambda$%
	\or $\mu$%
	\or $\nu$%
	\or $\xi$%
	\or $o$%
	\or $\pi$%
	\or $\rho$%
	\or $\sigma$%
	\or $\tau$%
	\or $\upsilon$%
	\or $\phi$%
	\or $\chi$%
	\or $\psi$%
	\or $\omega$%
\fi}
\def\@Greek#1{\ifcase#1
	\or $\mathrm{A}$%
	\or $\mathrm{B}$%
	\or $\Gamma$%
	\or $\Delta$%
	\or $\mathrm{E}$%
	\or $\mathrm{Z}$%
	\or $\mathrm{H}$%
	\or $\Theta$%
	\or $\mathrm{I}$%
	\or $\mathrm{K}$%
	\or $\Lambda$%
	\or $\mathrm{M}$%
	\or $\mathrm{N}$%
	\or $\Xi$%
	\or $\mathrm{O}$%
	\or $\Pi$%
	\or $\mathrm{P}$%
	\or $\Sigma$%
	\or $\mathrm{T}$%
	\or $\mathrm{Y}$%
	\or $\Phi$%
	\or $\mathrm{X}$%
	\or $\Psi$%
	\or $\Omega$%
\fi}

\AddEnumerateCounter{\greek}{\@greek}{24}
\AddEnumerateCounter{\Greek}{\@Greek}{12}

\makeatother

\let\polishlcross=\l
\def\l{\ifmmode\ell\else\polishlcross\fi}

\def\paragraph#1{%
  \noindent\textbf{#1.}\enspace}

\let\emptyset=\varnothing

\makeatletter
\def\moverlay{\mathpalette\mov@rlay}
\def\mov@rlay#1#2{\leavevmode\vtop{   \baselineskip\z@skip \lineskiplimit-\maxdimen
   \ialign{\hfil$\m@th#1##$\hfil\cr#2\crcr}}}
\newcommand{\charfusion}[3][\mathord]{
    #1{\ifx#1\mathop\vphantom{#2}\fi
        \mathpalette\mov@rlay{#2\cr#3}
      }
    \ifx#1\mathop\expandafter\displaylimits\fi}
\makeatother

\DeclareFontFamily{U}  {MnSymbolC}{}
\DeclareSymbolFont{MnSyC}         {U}  {MnSymbolC}{m}{n}
\DeclareFontShape{U}{MnSymbolC}{m}{n}{
    <-6>  MnSymbolC5
   <6-7>  MnSymbolC6
   <7-8>  MnSymbolC7
   <8-9>  MnSymbolC8
   <9-10> MnSymbolC9
  <10-12> MnSymbolC10
  <12->   MnSymbolC12}{}
\DeclareMathSymbol{\powerset}{\mathord}{MnSyC}{180}

\usepackage{tikz}
\usetikzlibrary{calc,decorations.pathmorphing}
\pgfdeclarelayer{background}
\pgfdeclarelayer{foreground}
\pgfdeclarelayer{front}
\pgfsetlayers{background,main,foreground,front}

\let\epsilon=\varepsilon

\let\rho=\varrho
\let\theta=\vartheta
\let\kappa=\varkappa

\let\E=\EE

\def\PP{{\mathds P}}

\newcommand{\cN}{\mathcal{N}}

\theoremstyle{plain}
\newtheorem{thm}{Theorem}[section]
\newtheorem{theorem}[thm]{Theorem}

\newtheorem{prop}[thm]{Proposition}
\newtheorem{claim}[thm]{Claim}

\newtheorem{lemma}[thm]{Lemma}

\theoremstyle{definition}

\usepackage{accents}

\let\phi=\varphi

\begin{document}
\dedicatory{Dedicated to the memory of Ron Graham}

\title[Twins and up-and-down permutations]{On weak twins and up-and-down sub-permutations}

\author{Andrzej Dudek}
\address{Department of Mathematics, Western Michigan University, Kalamazoo, MI, USA}
\email{\tt andrzej.dudek@wmich.edu}
\thanks{The first author was supported in part by Simons Foundation Grant \#522400.}

\author{Jaros\l aw Grytczuk}
\address{Faculty of Mathematics and Information Science, Warsaw University of Technology, Warsaw, Poland}
\email{j.grytczuk@mini.pw.edu.pl}
\thanks{The second author was supported in part by the Polish NSC grant 2015/17/B/ST1/02660.}

\author{Andrzej Ruci\'nski}
\address{Department of Discrete Mathematics, Adam Mickiewicz University, Pozna\'n, Poland}
\email{\tt rucinski@amu.edu.pl}
\thanks{The third author was supported in part by the Polish NSC grant 2018/29/B/ST1/00426}

\begin{abstract}
 Two permutations $(x_1,\dots,x_w)$ and $(y_1,\dots,y_w)$ are \emph{weakly similar} if $x_i<x_{i+1}$ if and only if $y_i<y_{i+1}$ for all $1\leqslant i \leqslant w$. Let $\pi$ be a permutation of the set $[n]=\{1,2,\dots, n\}$ and let $wt(\pi)$ denote the largest integer $w$ such that $\pi$ contains a pair of \emph{disjoint} weakly similar sub-permutations (called \emph{weak twins}) of length $w$. Finally, let $wt(n)$ denote the minimum of $wt(\pi)$ over all permutations $\pi$ of $[n]$. Clearly, $wt(n)\le n/2$. In this paper we show that $\tfrac n{12}\le wt(n)\le\tfrac n2-\Omega(n^{1/3})$.

 We also study a variant of this problem. Let us say that $\pi'=(\pi(i_1),...,\pi(i_j))$, $i_1<\cdots<i_j$, is an \emph{alternating} (or \emph{up-and-down}) sub-permutation of $\pi$ if $\pi(i_1)>\pi(i_2)<\pi(i_3)>...$ or $\pi(i_1)<\pi(i_2)>\pi(i_3)<...$. Let $\Pi_n$ be a random permutation selected uniformly from all $n!$ permutations of $[n]$. It is known (\cite{StanleyMichigan}) that the length of a longest alternating permutation in $\Pi_n$ is asymptotically almost surely (a.a.s.) close to $2n/3$. We study the maximum length $\alpha(n)$ of a pair of disjoint alternating sub-permutations in $\Pi_n$ and show that there are two constants $1/3<c_1<c_2<1/2$ such that a.a.s. $c_1n\le \alpha(n)\le c_2n$.
 
 In addition, we show that the alternating shape is the most popular among all permutations of a given length.
\end{abstract}

\maketitle

%\tableofcontents

\setcounter{footnote}{1}

\section{Introduction}
Looking for twin objects in mathematical structures has long and rich tradition going back to ancient geometric dissection problems and culminating in the famous Banach-Tarski Paradox (see \cite{TomkowiczWagon}). From that research we know, for instance, that two very differently looking objects, like the Sun and an apple, or the square and the circle, can be split into finitely many pairwise identical pieces. A general problem is to partition a given structure (or structures) into possibly few pairwise similar substructures. A related issue is to find, in a given structure, a pair of twin substructures, as large as possible.

 Despite such `continuous' origins, questions of that sort can be studied in diverse discrete contexts, with various types of similarity specified between the objects. For instance, Chung, Graham, Erd\H {o}s, Ulam, and Yao \cite{ChungEGUY} studied edge decompositions of pairs of graphs into pairwise isomorphic subgraphs (see also \cite{ChungEG}, \cite{Graham}), while Erd\H {o}s, Pach, and Pyber \cite{ErdosPachPyber} looked for twins in a single graph (defined as a pair of edge disjoint isomorphic subgraphs). Axenovich, Person, and Puzynina (\cite{Axenovich}, \cite{AxenovichPersonPuzynina}) investigated twins in words, and Gawron \cite{Gawron}, inspired by their work, initiated exploration of twins in permutations (defined as a pair of disjoint \emph{order-isomorphic} sub-permutations).

Let us dwell on this last problem for a while. By a \emph{permutation} we mean any finite sequence of distinct positive integers. Let $t(n)$ be the maximum number $k$ such that every permutation of length $n$ has a pair of twins, each of length $k$. By a probabilistic argument, Gawron \cite{Gawron} proved that $t(n)=O(n^{2/3})$ and made a conjecture that this is best possible, that is, $t(n)=\Theta(n^{2/3})$. We confirmed this conjecture in \cite{DGR} (up to a logarithmic factor) for a \emph{random} permutation. A refinement of our result (getting rid of the logarithmic factor) was then obtained by Bukh and Rudenko \cite{BukhR}. In the deterministic case, the $t(n)\geqslant\Omega(\sqrt{n})$ follows immediately from the famous result of Erd\H {o}s and Szekeres \cite{ErdosSzekeres} on monotone subsequences in permutations. Currently, the best lower bound $t(n)=\Omega(n^{3/5})$ is due to Bukh and Rudenko \cite{BukhR}.

In this paper we consider a weaker type of similarity of permutations than order-isomorphism in which we only look at the relations between neighboring element. We say that two permutations $(x_1,\dots,x_w)$ and $(y_1,\dots,y_w)$ are \emph{weakly similar} if $x_i<x_{i+1}$ if and only if $y_i<y_{i+1}$ for all  $1\leqslant i\leqslant w$. 

This notion can be equivalently defined in terms of shapes. For our purposes, \emph{the shape} of a permutation $\pi=(x_1,\dots,x_w)$  is defined as a binary sequence $s(\pi)=(s_1,\dots,s_{w-1})$ with elements from the set $\{+,-\}$, where $s_i=+$ if and only $x_i<x_{i+1}$, $i=1,\dots,w-1$. For instance, $s(6, 1 , 4 , 3, 7 , 9, 8 , 2 , 5)=(-,+,-,+,+,-,-,+)$.
Then, permutations $\pi_x=(x_1,\dots,x_w)$ and $\pi_y=(y_1,\dots,y_w)$ are weakly similar if $s(\pi_x)=s(\pi_y)$.

%The (binary) sequence of the signs $<,>$ between consecutive elements of each sequence is called \emph{the pattern} of each of the two permutations.

%Note that given a permutation $(x_1,\dots,x_k)$ and a $k$-element set $y$ of positive integers, there is only one permutation of $y$ similar to $(x_1,\dots,x_k)$.

Let $[n]=\{1,2,\dots,n\}$ and let $\pi$ be a permutation of $[n]$, called also an \emph{$n$-permutation}. Two weakly similar disjoint sub-permutations of $\pi$ are called \emph{weak twins} and the \emph{length of the twins} is defined as the number of elements in just \emph{one} of the sub-permutations. For example, in permutation $$(6,\colorbox{cyan}{1},\colorbox{cyan}{4},3,\colorbox{Lavender}{7},9,\colorbox{Lavender}{8},\colorbox{cyan}{2},\colorbox{Lavender}{5}),$$
the blue $(1,4,2)$ and red $(7,8,5)$ subsequences form  weak twins of length $3$ (with a common shape $(+,-)$).

Let $wt(\pi)$ denote the largest integer $w$ such that $\pi$ contains  weak twins of length $w$. Further, let $wt(n)$ denote the minimum of $wt(\pi)$ over all $n$-permutations $\pi$. In other words, $wt(n)$ is the largest integer $w$ such that every $n$-permutation contains  weak twins of length~$w$.
Our aim is to estimate this function which, unlike its stronger version $t(n)$, turns out to be linear in $n$.

%along with some of its variants subject to various restrictions, for all permutations, as well as, for almost all permutations, that is, for the random permutation $\Pi_n$, selected uniformly  from all $n!$ permutations of $[n]$.

%We begin with bounds on $wt(n)$. The proof of the following result will be given in the next section.

%Since, in general, it seems to be  very hard to estimate $wt(n)$, in this paper we turn to a special, alternating instance of it.

\begin{theorem}\label{thm_weak} For $n$ large enough,
\begin{equation}\label{bounds_gen}
\frac n{12}\le wt(n)\le\frac n2-\Omega(n^{1/3}).
\end{equation}
\end{theorem}

Turning to our second result, note that given a sequence $s^{(n)}$ of length $n-1$, it is quite nontrivial to determine the number $N(s^{(n)})$ of $n$-element permutations with the shape~$s^{(n)}$. Of course, there is just one permutation with a given monotone shape, $(+,\dots,+)$ and $(-,\dots,-)$. But already for the alternating shapes, $a^{(n)}_{+}=(+,-,+,\dots)$ and $a^{(n)}_{-}=(-,+,-,\dots)$, this is so called Andr\'e's problem \cite{Andre}, which was solved asymptotically in the 19th century and exactly, in terms of a finite  sum of Stirling numbers, only in the 21th century \cite{Mendes} (see also \cite{StanleySurvey}).

 The asymptotic formula of Andr\'e says that, setting $A_n:=N(a^{(n)}_{+})=N(a^{(n)}_{-})$,
 $$A_n\sim2(2/\pi)^{n+1}n!.$$ In other words, the probability that a random $n$-permutation $\Pi_n$ is alternating (either way) is only $\sim4(2/\pi)^{n+1}$.
 %, so it is quite unlikely that $\Pi_n$ contains weak twins of length very close $n/2$ with an alternating pattern.
   On the other hand,  by the result of Stanley \cite{StanleyMichigan}, we know that a.a.s. a random $n$-permutation contains an alternating subsequence of length at least $\sim 2n/3$, yielding alternating twins of length at least $\sim n/3$ (just split in half a longest alternating sub-permutation in $\Pi_n$). In Theorem \ref{lem_ran_alt} we show, however, that a.a.s. one can get  substantially longer alternating twins; on the other hand, they are much shorter than $n/2$, the absolute upper bound.

To state this result, let $\alpha(\pi)$ be the largest integer $w$ such that $\pi$ contains weak twins of length $w$ with an alternating shape, $a^{(w)}_{+}$ or $a^{(w)}_{-}$. We will call them \emph{alternating twins}.
Further, set $\alpha_n:=\alpha(\Pi_n)$, where $\Pi_n$ is a random $n$-permutation.

\begin{theorem}\label{lem_ran_alt}
	A.a.s.
\begin{equation}\label{bounds}
\left(\frac13+\frac{1}{60}+o(1)\right) n\le \alpha_n\le\left(\frac12-\frac{1}{120}+o(1)\right) n.
\end{equation}
\end{theorem}

We end this paper by proving that, in fact, permutations with alternating shapes are the most popular ones. This result, not directly related to our  main theorems, may be of independent interest. %Set $A_n=N(p_n^{<})=N(p_n^{>})$.

\begin{prop}\label{alt_the_king}
For every $n$ and every shape $s^{(n)}$ of length $n-1$, we have $N(s^{(n)})\le A_n$.
\end{prop}
%We believe that the number of permutations realizing a given pattern is the largest just when the pattern is  alternating.
The proof of Proposition \ref{alt_the_king} can be found in Section \ref{proof_prop}.

\subsection*{Note} We believe that Ron Graham would like the topic of this paper. Not only he was among those who planted the idea  of twins into the combinatorial soil, but he also wrote several papers devoted to permutations , both, with and without connections to juggling
(see, e.g., \texttt{http://www.math.ucsd.edu/\~{}ronspubs/} for the entire collection of Ron's publications).

\section{Proofs of Theorems \ref{thm_weak} and \ref{lem_ran_alt}}

\subsection{Extremal points}
In our proofs a decisive role is played by local extremes. We call the element $i$ \emph{maximal} in $\pi$ if $i=1$ and $\pi(1)>\pi(2)$, or $i=n$ and $\pi(n-1)<\pi(n)$, or $1<i<n$ and $\pi(i-1)<\pi(i)>\pi(i+1)$. By swapping all signs $<$ and $>$ around, we obtain the notion of \emph{a minimal} point $i$ in $\pi$. Maximal and minimal points alternate and are jointly referred to as \emph{extremal}. The points $1$ and $n$ are always extremal. Clearly, all extremal points of $\pi$ form an alternating sequence in $\pi$. In fact, as shown by B\'{o}na (see \cite{StanleySurvey}, and \cite{HoudreR} for a proof), it is the longest one.

Let $E=\{j_1,\dots,j_k\}$ be the set of the extremal points in $\pi$.
 These points divide the whole range $[n]$ into monotone segments
$$\pi_i=(\pi(j_i),\pi(j_i+1),\dots,\pi(j_{i+1})),$$
$i=1,\dots,k-1$.

\subsection{Weak twins}

\begin{proof}[Proof of Theorem~\ref{thm_weak}, lower bound]
 For the lower bound, recall that the extremal points of a permutation $\pi$ partition it into monotone segments $\pi_1,\dots,\pi_{k-1}$.
 As the extremal points themselves form an alternating sub-sequence $E$ of $\pi$, by splitting it evenly, we obtain a pair of weak twins of length $\lfloor k/2\rfloor$. Thus, we may assume that $k-1\le n/6$, since otherwise $\lfloor k/2\rfloor\ge k/2-1/2\ge n/12$ and we are done.

 Let $Q_1,\dots, Q_\ell$ be those segments among $\pi_1,\dots,\pi_{k-1}$ which contain at least 4 elements each. It is easy to check that
 $$|Q_1|+\cdots+|Q_\ell|\ge \frac{1}{2}n.$$
 Indeed, otherwise we would have
 $$n=\sum_{i=1}^{k-1}|\pi_i|< 3(k-1)+\frac{1}{2}n\le n,$$
 a contradiction.
 %Let us now select a pair of (monotone) twins $(A_i,B_i)$ in each $Q_i$ of maximal length $\lfloor|Q_i|/2\rfloor$, by assigning the elements of $Q_i$ alternately to $A_i$ and $B_i$ (from left to right).
 All we need now is the following proposition.

 \begin{prop}\label{glue}
 One can find  weak twins in $Q_1\cup \cdots \cup Q_\ell$ of length at least
 $$\frac12\sum_{i=1}^\ell|Q_i|-\ell.$$
 \end{prop}

 Before proving the proposition, let us finish the proof of the lower bound in \eqref{bounds_gen}. By Proposition \ref{glue}, there is in $\pi$ a pair of weak twins of length

 \begin{align*}
 \frac12\sum_{i=1}^\ell|Q_i|-\ell\ge\frac{1}{4}n-(k-1)\ge \frac{1}{4}n-\frac{1}{6}n=\frac n{12}.
 \end{align*}

\end{proof}

\begin{proof}[Proof of Proposition~\ref{glue}]  We begin with the following observation. We say that weak twins $(A,B)$, where $A=(\pi(i_1),\dots,\pi(i_k))$, $i_1<\cdots<i_k$, and $B=(\pi(j_1),\dots,\pi(j_k))$, $j_1<\cdots<j_k$, are \emph{aligned upward}, respectively, \emph{downward} if the  two right-most elements of $A$ and the  two right-most elements of $B$ interwind and form a monotone sub-sequence, that is, $j_{k-1}<i_{k-1}<j_k<i_k$ and $\pi(j_{k-1})<\pi(i_{k-1})<\pi(j_k)<\pi(i_k)$, or, respectively, $\pi(j_{k-1})>\pi(i_{k-1})>\pi(j_k)>\pi(i_k)$.

\begin{claim}\label{cl_glue}
Let $(A,B)$ be  aligned weak twins in $\pi$  and let $Q=(\pi(m_1),\dots,\pi(m_s))$, $s\ge4$, be a monotone sub-sequence of $\pi$ completely to the right of $(A,B)$, that is, $m_1>i_k$. Then one can extend $(A,B)$ to a new pair of aligned weak twins $(A',B')$ which contains all elements of $A,B$ and $Q$ except  for at most 2 elements. The lost elements are either all from $Q$ (the first or the last or both) or one from $Q$ (the last one) and one from $A$ (the last one).
\end{claim}

Proposition~\ref{glue} follows quickly from the above claim. Indeed, by its repeated applications, beginning with selecting a pair of aligned weak twins $(A_1,B_1)$ within $Q_1$ (here we lose one element in the case when $|Q_1|$ is odd), we recursively construct the desired object losing along the way at most $1+2(\ell-1)<2\ell$ elements.
\end{proof}

\begin{proof}[Proof of Claim~\ref{cl_glue}] W.l.o.g., assume that the weak twins $(A,B)$ are aligned upward. However, with respect to $Q$,  we  have to consider both cases of its monotonicity. We first assume that $Q$ is  increasing.

We are going to examine 4 cases of how  the two bottom  values in $Q$ position themselves with respect to the two top ones in $(A,B)$ (see Figure~\ref{fig:claim:gluing}). Set $a=\pi(i_k)$, $\bar a=\pi(i_{k-1})$, $b=\pi(j_k)$, $\bar b=\pi(j_{k-1})$, and  $q_i=\pi(m_j)$, $j=1,2,\dots,s$. Recall that $\bar b<\bar a<b<a$.

\begin{figure}
\captionsetup[subfigure]{labelformat=empty}
\begin{center}

\begin{subfigure}[b]{0.48\textwidth}
\scalebox{0.9}
{
\centering
\begin{tikzpicture}
[line width = .5pt,
vtx/.style={circle,draw,black,very thick,fill=black, line width = 1pt, inner sep=2pt},
]

    \node[vtx] (a) at (3,3) {};
    \node[vtx] (a_bar) at (1,0.8) {};
    \node[vtx] (b) at (2.4,1.2) {};
    \node[vtx] (b_bar) at (0,0) {};
    \fill[fill=black, outer sep=1mm]  (a) circle (0.1) node [left] {$a$};
    \fill[fill=black, outer sep=1mm]  (a_bar) circle (0.1) node [left] {$\bar{a}$};
    \fill[fill=black, outer sep=1mm]  (b) circle (0.1) node [below] {$b$};
    \fill[fill=black, outer sep=1mm]  (b_bar) circle (0.1) node [below] {$\bar{b}$};
    \draw[line width=0.5mm, color=black, outer sep=2mm]   (a) -- (a_bar) node[pos=0.6, above] {$A$};
    \draw[line width=0.5mm, color=black, outer sep=1mm]   (b) -- (b_bar) node[pos=0.4, below] {$B$};

    \node[inner sep=0pt] (aq) at (6,3) {};
    \node[inner sep=0pt] (bq) at (6,1.2) {};
    \draw[line width=0.3mm, color=lightgray]  (a) -- (aq);
    \draw[line width=0.3mm, color=lightgray]  (b) -- (bq);

    \node[vtx] (q1) at (4, 0.2) {};
    \node[vtx] (q2) at (4.5, 2) {};
    %\node[vtx] (q3) at (5, 3.8) {};
    \node[vtx] (qs) at (5.75, 6.5) {};

    \fill[fill=black, outer sep=1mm]  (q1) circle (0.1) node [right] {$q_1$};
    \fill[fill=black, outer sep=1mm]  (q2) circle (0.1) node [right] {$q_2$};
    %\fill[fill=black, outer sep=1mm]  (q3) circle (0.1) node [right] {$q_3$};
    \fill[fill=black, outer sep=1mm]  (qs) circle (0.1) node [right] {$q_s$};
    \draw[line width=0.5mm, color=black]  (q1) -- (q2);
    \draw[line width=0.5mm, color=black, dashed]  (q2) -- (qs) node[pos=0.5, right] {$Q$};

    \draw[line width=0.7mm, color=blue]   (a) -- (q2);
    \draw[line width=0.7mm, color=blue]   (b) -- (q1);

\end{tikzpicture}
}
\caption{Case 1}
\end{subfigure}
     \
\begin{subfigure}[b]{0.48\textwidth}
\scalebox{0.9}
{
\centering
\begin{tikzpicture}
[line width = .5pt,
vtx/.style={circle,draw,black,very thick,fill=black, line width = 1pt, inner sep=2pt},
]

    \node[vtx] (a) at (3,3) {};
    \node[vtx] (a_bar) at (1,0.8) {};
    \node[vtx] (b) at (2.4,1.2) {};
    \node[vtx] (b_bar) at (0,0) {};
    \fill[fill=black, outer sep=1mm]  (a) circle (0.1) node [left] {$a$};
    \fill[fill=black, outer sep=1mm]  (a_bar) circle (0.1) node [left] {$\bar{a}$};
    \fill[fill=black, outer sep=1mm]  (b) circle (0.1) node [below] {$b$};
    \fill[fill=black, outer sep=1mm]  (b_bar) circle (0.1) node [below] {$\bar{b}$};
    \draw[line width=0.5mm, color=black, outer sep=2mm]   (a) -- (a_bar) node[pos=0.6, above] {$A$};
    \draw[line width=0.5mm, color=black, outer sep=1mm]   (b) -- (b_bar) node[pos=0.4, below] {$B$};

    \node[inner sep=0pt] (aq) at (6,3) {};
    \node[inner sep=0pt] (bq) at (6,1.2) {};
    \draw[line width=0.3mm, color=lightgray]  (a) -- (aq);
    \draw[line width=0.3mm, color=lightgray]  (b) -- (bq);

    \node[vtx] (q1) at (4, 0.2) {};
    \node[vtx] (q2) at (5, 3.8) {};
    \node[vtx] (q3) at (5.5, 5.6) {};
    \node[vtx] (qs) at (5.75, 6.5) {};

    \fill[fill=black, outer sep=1mm]  (q1) circle (0.1) node [right] {$q_1$};
    \fill[fill=black, outer sep=1mm]  (q2) circle (0.1) node [right] {$q_2$};
    \fill[fill=black, outer sep=1mm]  (q3) circle (0.1) node [right] {$q_3$};
    \fill[fill=black, outer sep=1mm]  (qs) circle (0.1) node [right] {$q_s$};
    \draw[line width=0.5mm, color=black]  (q1) -- (q3) node[pos=0.85, right] {$Q$};;
   \draw[line width=0.5mm, color=black, dashed]  (q3) -- (qs);

    \draw[line width=0.7mm, color=blue]   (a) -- (q3);
    \draw[line width=0.7mm, color=blue]   (b) -- (q2);

\end{tikzpicture}
}
\caption{Case 2}
\end{subfigure}

\begin{subfigure}[b]{0.48\textwidth}
\scalebox{0.9}
{
\centering
\begin{tikzpicture}
[line width = .5pt,
vtx/.style={circle,draw,black,very thick,fill=black, line width = 1pt, inner sep=2pt},
]

    \node[vtx] (a) at (3,3) {};
    \node[vtx] (a_bar) at (1,0.8) {};
    \node[vtx] (b) at (2.4,1.2) {};
    \node[vtx] (b_bar) at (0,0) {};
    \fill[fill=black, outer sep=1mm]  (a) circle (0.1) node [left] {$a$};
    \fill[fill=black, outer sep=1mm]  (a_bar) circle (0.1) node [left] {$\bar{a}$};
    \fill[fill=black, outer sep=1mm]  (b) circle (0.1) node [below] {$b$};
    \fill[fill=black, outer sep=1mm]  (b_bar) circle (0.1) node [below] {$\bar{b}$};
    \draw[line width=0.5mm, color=black, outer sep=2mm]   (a) -- (a_bar) node[pos=0.6, above] {$A$};
    \draw[line width=0.5mm, color=black, outer sep=1mm]   (b) -- (b_bar) node[pos=0.4, below] {$B$};

    \node[inner sep=0pt] (aq) at (6,3) {};
    \node[inner sep=0pt] (bq) at (6,1.2) {};
    \draw[line width=0.3mm, color=lightgray]  (a) -- (aq);
    \draw[line width=0.3mm, color=lightgray]  (b) -- (bq);

    \node[vtx] (q1) at (4.5, 2) {};
    \node[vtx] (q2) at (5, 3.8) {};
    \node[vtx] (qs) at (5.75, 6.5) {};

    \fill[fill=black, outer sep=1mm]  (q1) circle (0.1) node [right] {$q_1$};
    \fill[fill=black, outer sep=1mm]  (q2) circle (0.1) node [right] {$q_2$};
    %\fill[fill=black, outer sep=1mm]  (q3) circle (0.1) node [right] {$q_3$};
    \fill[fill=black, outer sep=1mm]  (qs) circle (0.1) node [right] {$q_s$};
    \draw[line width=0.5mm, color=black]  (q1) -- (q2);
    \draw[line width=0.5mm, color=black, dashed]  (q2) -- (qs) node[pos=0.5, right] {$Q$};;

    \draw[line width=0.7mm, color=blue]   (a) -- (q2);
    \draw[line width=0.7mm, color=blue]   (b) -- (q1);

\end{tikzpicture}
}
\caption{Case 3}
\end{subfigure}
\
\begin{subfigure}[b]{0.48\textwidth}
\scalebox{0.9}
{
\centering
\begin{tikzpicture}
[line width = .5pt,
vtx/.style={circle,draw,black,very thick,fill=black, line width = 1pt, inner sep=2pt},
]

    \node[vtx] (a) at (3,3) {};
    \node[vtx] (a_bar) at (1,0.8) {};
    \node[vtx] (b) at (2.4,1.2) {};
    \node[vtx] (b_bar) at (0,0) {};
    \fill[fill=black, outer sep=1mm]  (a) circle (0.1) node [left] {$a$};
    \fill[fill=black, outer sep=1mm]  (a_bar) circle (0.1) node [left] {$\bar{a}$};
    \fill[fill=black, outer sep=1mm]  (b) circle (0.1) node [below] {$b$};
    \fill[fill=black, outer sep=1mm]  (b_bar) circle (0.1) node [below] {$\bar{b}$};
    \draw[line width=0.5mm, color=black, outer sep=2mm]   (a) -- (a_bar) node[pos=0.6, above] {$A$};
    \draw[line width=0.5mm, color=black, outer sep=1mm]   (b) -- (b_bar) node[pos=0.4, below] {$B$};

    \node[inner sep=0pt] (aq) at (6,3) {};
    \node[inner sep=0pt] (bq) at (6,1.2) {};
    \draw[line width=0.3mm, color=lightgray]  (a) -- (aq);
    \draw[line width=0.3mm, color=lightgray]  (b) -- (bq);

%    \node[vtx] (q1) at (4.5, 2) {};
    \node[vtx] (q1) at (4.25+0.125, 1.1+0.45) {};
    \node[vtx] (q2) at (4.625, 2.45) {};
    \node[vtx] (qs) at (5.75, 6.5) {};

    \fill[fill=black, outer sep=1mm]  (q1) circle (0.1) node [right] {$q_1$};
    \fill[fill=black, outer sep=1mm]  (q2) circle (0.1) node [right] {$q_2$};
    %\fill[fill=black, outer sep=1mm]  (q3) circle (0.1) node [right] {$q_3$};
    \fill[fill=black, outer sep=1mm]  (qs) circle (0.1) node [right] {$q_s$};
    \draw[line width=0.5mm, color=black]  (q1) -- (q2);
    \draw[line width=0.5mm, color=black, dashed]  (q2) -- (qs) node[pos=0.5, right] {$Q$};;

    \draw[line width=0.7mm, color=blue]   (b) -- (q2);
    \draw[line width=0.7mm, color=blue]   (a_bar) .. controls (3,2) .. (q1);

\end{tikzpicture}
}
\caption{Case 4}
\end{subfigure}

\end{center}

\caption{Extending twins in the proof of Claim~\ref{cl_glue} with increasing $Q$.}
\label{fig:claim:gluing}
\end{figure}
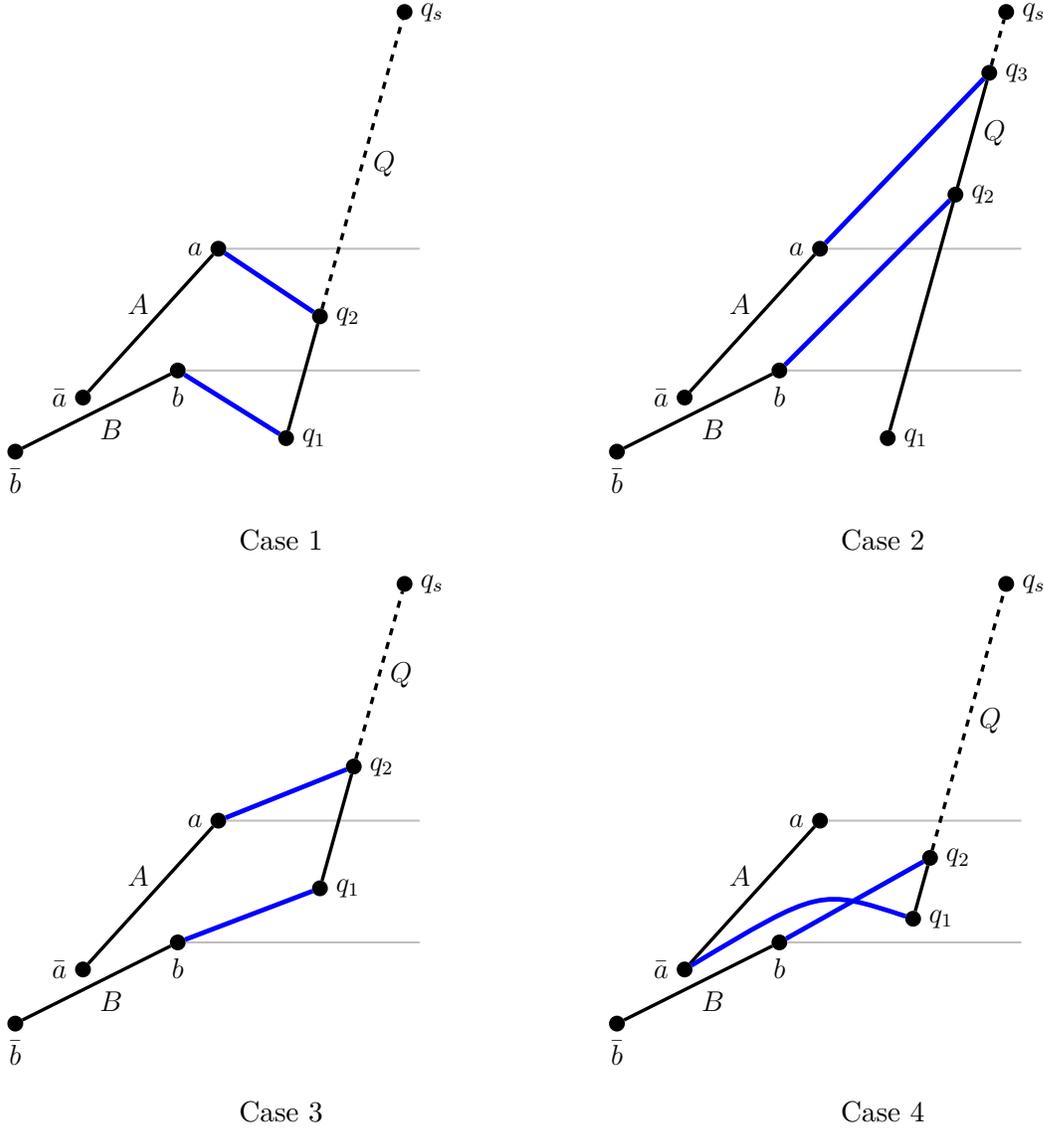

\textbf{Case 1: $q_1<b, q_2<a$.} We extend $A$ and $B$ as follows:
$$A'=A,q_2,q_4,\dots,\quad B'=B,q_1,q_3,\dots.$$
If $s$ is odd, the point $q_s$ is not used (we say it is lost). Note that due to the order of $q_1,q_2,q_3,q_4$, the new pair $(A',B')$ is indeed aligned.

\textbf{Case 2: $q_1<b<a<q_2$.} Here we set
$$A'=A,q_3,q_5,\dots,\quad B'=B,q_2,q_4,\dots.$$
We definitely lose $q_1$ and, if $s$ is even, we also lose $q_s$. For $s=4$ or $s=5$, the last 4 points of $(A',B')$ are thus $b,a,q_2,q_3$, which  are aligned upward. If $s\ge 6$, then $(A',B')$ is aligned as well.

\textbf{Case 3: $q_1>b, q_2>a$.} This case is very similar to Case 1, so we omit the details.

\textbf{Case 4: $b<q_1<q_2<a$.} This is the only case when we lose a point of $(A,B)$. Let $A^-$ denote the sub-sequence $A$ without the last element, $a$. We set
$$A'=A^-,q_1,q_3,\dots,\quad B'=B,q_2,q_4,\dots.$$
Besides $a$, we may also lose $q_s$, provided $s$ is even. Observe that for $s=4$, $b,q_1,q_2,q_3$ are aligned upward.
This exhaust the case when $Q$ is increasing.

For decreasing $Q$, there are also 4 cases to examine. However, three of them, namely,  (i) $a>q_1, b>q_2$,
  (ii) $q_1>a>b>q_2$, and (iii) $q_1>a, q_2>b$ are very similar to those for increasing $Q$, so we leave them for the reader. The only somewhat different case is when (iv) $a>q_1>q_2>b$ (see Figure~\ref{fig:claim:gluing:decr}). Then, denoting by $B^-$ the sub-sequence $B$ without its last element, $b$, we set
$$A'=A,q_2,q_4,\dots,\quad B'=B^-,q_1,q_3,\dots.$$
Besides $b$, we may also lose $q_s$, provided $s$ is even. Finally, observe that for $s=4$, $a,q_1,q_2,q_3$ are aligned downward, though with the roles of $A'$ and $B'$ switched (which does not really matter to us; formally we should swap $A'$ and $B'$ around).

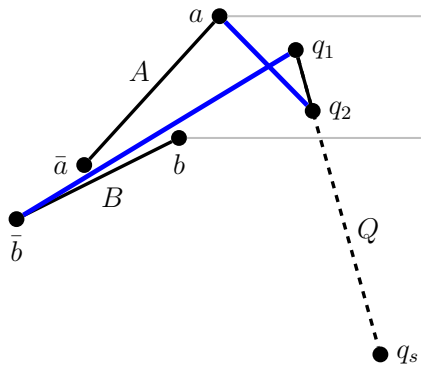
\begin{figure}
\captionsetup[subfigure]{labelformat=empty}
\begin{center}

\scalebox{0.9}
{
\centering
\begin{tikzpicture}
[line width = .5pt,
vtx/.style={circle,draw,black,very thick,fill=black, line width = 1pt, inner sep=2pt},
]

    \node[vtx] (a) at (3,3) {};
    \node[vtx] (a_bar) at (1,0.8) {};
    \node[vtx] (b) at (2.4,1.2) {};
    \node[vtx] (b_bar) at (0,0) {};
    \fill[fill=black, outer sep=1mm]  (a) circle (0.1) node [left] {$a$};
    \fill[fill=black, outer sep=1mm]  (a_bar) circle (0.1) node [left] {$\bar{a}$};
    \fill[fill=black, outer sep=1mm]  (b) circle (0.1) node [below] {$b$};
    \fill[fill=black, outer sep=1mm]  (b_bar) circle (0.1) node [below] {$\bar{b}$};
    \draw[line width=0.5mm, color=black, outer sep=2mm]   (a) -- (a_bar) node[pos=0.6, above] {$A$};
    \draw[line width=0.5mm, color=black, outer sep=1mm]   (b) -- (b_bar) node[pos=0.4, below] {$B$};

    \node[inner sep=0pt] (aq) at (6,3) {};
    \node[inner sep=0pt] (bq) at (6,1.2) {};
    \draw[line width=0.3mm, color=lightgray]  (a) -- (aq);
    \draw[line width=0.3mm, color=lightgray]  (b) -- (bq);

%    ratio = 0.25 / 0.9
    %\node[vtx] (q1) at (4.25+0.125, 1.1+0.45) {};
    %\node[vtx] (q2) at (4.625, 2.45) {};
    %\node[vtx] (qs) at (5.75, 6.5) {};

    \node[vtx] (q1) at (4.125, 2.5) {};
    \node[vtx] (q2) at (4.125+0.25, 2.5 - 0.9) {};
    \node[vtx] (qs) at (4.125+0.25+1, 2.5 - 0.9-3.6) {};

   \fill[fill=black, outer sep=1mm]  (q1) circle (0.1) node [right] {$q_1$};
   \fill[fill=black, outer sep=1mm]  (q2) circle (0.1) node [right] {$q_2$};

    \fill[fill=black, outer sep=1mm]  (qs) circle (0.1) node [right] {$q_s$};
    \draw[line width=0.5mm, color=black]  (q1) -- (q2);
    \draw[line width=0.5mm, color=black, dashed]  (q1) -- (qs) node[pos=0.6, right] {$Q$};;

    \draw[line width=0.7mm, color=blue]   (b_bar) -- (q1);
    \draw[line width=0.7mm, color=blue]   (a) -- (q2);

\end{tikzpicture}
}

\end{center}

\caption{Extending twins in the proof of Claim~\ref{cl_glue} with decreasing $Q$.}
\label{fig:claim:gluing:decr}
\end{figure}

\end{proof}

\bigskip
\begin{proof}[Proof of Theorem~\ref{thm_weak}, upper bound] We are going to construct a permutation $\pi$ on $[n]$, $n$ large enough, with no weak twins longer than $n/2-cn^{1/3}$ for some $c>0$. This permutation will consist of  $k'\le k:=\lceil n^{1/3}\rceil$ consecutive increasing segments $P_1,\dots,P_{k'}$ with $\max P_{i+1}<\min P_i$, of diminishing lengths which have to be chosen carefully. For $i=1,\dots,k$, set
$$x_i=2k^2-2(i-1)k-1.$$
Note that for all $i=1,\dots,k$, $x_i>0$ and $x_i$ is an odd integer. Moreover,
$$\sum_{i=1}^kx_i=2k^3-2k\binom k2-k=k^3+k^2-k>n.$$
Let $k'=\min\{j: \sum_{i=1}^jx_i\ge n\}$. Then we set $|P_i|=x_i$, $i=1,\dots,k'-1$, and \newline $|P_{k'}|=n-\sum_{i=1}^{k'-1}x_i$.
Since $\sum_{i=1}^kx_i=n+O(k^2)$, with a big margin we have, say,  $k'\ge 0.99k$. Also, what is crucial here, for all $i=1,\dots,k'-1$, we have $x_i-x_{i+1}\ge 2k$, in fact, with equality except for $i=k'-1$.

So, we define $\pi=(P_1,\dots,P_{k'})$ in the following manner. We set
$$\pi(1)=n-x_1+1,\pi(2)=n-x_1+2,\dots,\pi(x_1)=n\quad\mbox{and}\quad A_1=(\pi(1),\dots,\pi(x_1)).$$
 Then we dip down and set
$$\pi(x_1+1)=n-x_1-x_2+1,\pi(x_1+2)=n-x_1-x_2+2,\dots,\pi(x_1+x_2)=n-x_1$$
and
$$ A_2=(\pi(x_1+1),\dots,\pi(x_1+x_2)),$$
 and so on, and so forth.

 We now state a proposition from which the desired bound follows quickly.

\begin{prop}\label{equal} Let $(A,B)$ be  weak twins in the permutation $\pi$ defined above of length $|A|=|B|\ge n/2-k/3$.
Then, for all $1\le i< k'$, we have $|A\cap P_i|=|B\cap P_i|$.
\end{prop}

Before proving the proposition, let us finish the proof of the upper bound in Theorem~\ref{thm_weak}. Suppose there is in $\pi$ a pair of weak twins of length at least $n/2-k/2$. Since for all $1\le i<k'$, $|P_i|$ is odd, in view of Proposition \ref{equal}, at least one point of each such $P_i$ is missing from $(A,B)$. Hence,
\[
|A|=|B|\le n/2-(k'-1)/2\le n/2 - (0.99k-1)/2<  n/2-k/3,
\]
a contradiction.
\end{proof}

\begin{proof}[Proof of Proposition~\ref{equal}] We proceed by (strong) induction on $i=1,\dots, k'-1$. Let us start with the base case $i=1$. Since at most $2k/3$ points of $\pi$ are not in $A\cup B$, while $|P_1|>2k/3$, w.l.o.g., $A\cap P_1\neq\emptyset$. It suffices to prove that also $B\cap P_1\neq\emptyset$, since then, due to the fact that the rest of $\pi$ lies totally below $P_1$, $A$ and $B$ must have the same number of elements in $P_1$. Suppose to the contrary that $B\cap P_1=\emptyset$. But then
$$|A\cap P_1|\ge|P_1|-\frac23k>|P_2|>|P_3|>\cdots,$$
so $A$ begins with a longer increasing segment than $B$ does, a contradiction with the notion of weak  twins.

For the induction step, which is similar to the base step, assume that $|A\cap P_j|=|B\cap B_j|$, for $j=1,\dots,i\le k'-2$. If $|A\cap P_{i+1}|=|B\cap B_{i+1}|=0$, then we are done.
W.l.o.g., assume that $|A\cap P_{i+1}|>0$. As before, it suffices to show that also $|B\cap P_{i+1}|>0$. Suppose otherwise.
Then, since at most $2k/3$ points of $\pi$ are not in $A\cup B$, we have
$$|A\cap P_{i+1}|\ge|P_{i+1}|-\frac23k>|P_{i+2}|>\cdots.$$
This means, however, that $A$ and $B$ will differ in the length of the first increasing segment commencing to the right of the point $\sum_{j=1}^ix_j$. This yields a contradiction with $(A,B)$ being weak twins and completes the proof.
\end{proof}

\subsection{Alternating weak twins} Recall that the extremal points of $\pi$ form an alternating sub-sequence. In the proof of the lower bound in  Theorem~\ref{lem_ran_alt}, we are going to use this fact and then reiterate it for the sub-permutation $\pi'$ obtained from $\pi$ by removing all the extremal points of $\pi$. As a crucial tool we invoke  the Azuma-Hoeffding inequality for random permutations (see, e.g., Lemma 11 in~\cite{FP} or  Section 3.2 in~\cite{McDiarmid98})
\begin{theorem}\label{azuma}
 Let $h(\pi)$ be a function of $n$-permutations such that if permutation $\pi_2$ is obtained from permutation $\pi_1$ by swapping two elements, then $|h(\pi_1)-h(\pi_2)|\le 1$.
Then, for every $\eta>0$,
\[
\PP(|h(\Pi_n)-\E[h(\Pi_n)]|\ge \eta)\le 2\exp(-\eta^2/(2n)).
\]
\end{theorem}

\begin{proof}[Proof of Theorem~\ref{lem_ran_alt}, lower bound]  We are going to show that extremal points are evenly distributed in both `halves' of $\Pi_n$. For mere convenience, we assume that $n$ is even.

Let $X_1$ and $X_2$ be the numbers of extremal points in $\Pi_n$ among, respectively, $\{1,\dots, n/2\}$ and $\{n/2+1,\dots,n\}$. Note that the probability that a given point $i$, $2\le i\le n-1$, is extremal is $2\times\tfrac13=\tfrac23$. Thus,
$$\E(X_1)=\E(X_2)=1+\left(\frac n2-1\right)\times\frac23=\frac{n+1}3.$$
Now we apply Theorem \ref{azuma} to show that this expectation is highly concentrated about its mean.
To verify the Lipschitz assumption, note that if $\pi_2$ is obtained from a permutation $\pi_1$ by swapping any two of its  elements, then trivially $|X_j(\pi_1)-X_j(\pi_2)|\le 6$, $j=1,2$. (A detailed analysis shows that 6 can be replaced by 4 which is optimal.)
% Np. (8,9,10,6,5,7,4,1,2,3) ma 6 extremow a zamieniajac 1 z 10 otrzymujemy (8,9,1,6,5,7,4,10,2,3), ktora ma 10 extremow
Consequently, Theorem~\ref{azuma} applied with $h(\pi)=X_j(\pi)/6$ and $\eta=n^{3/5}$ implies
\[
\PP(|X_j(\Pi_n)-\E[X_j(\Pi_n)]|\ge n^{3/5})=o(1)
\]
implying that a.a.s. $X_j=(1+o(1))\tfrac n3$, $j=1,2$.

It is quite hard to characterize the extremal points of $\pi'$. Unable to do so, we instead identify a 6-point configuration in $\pi$ which contains an extremal point of $\pi'$. A 6-tuple $\{i,i+1,i+2,i+3,i+4,i+5\}$, $1\le i\le n-5$, is called \emph{a lucky six} if $\pi(i)<\pi(i+1)< \pi(i+2)<\pi(i+3)>\pi(i+4)>\pi(i+5)$ and $\pi(i+2)>\pi(i+4)$, or when all signs $<$ and $>$ are swapped. It should be clear that in a lucky six  $i+3$ is an extremal point of $\pi$ and, most importantly, $i+2$ is an extremal point in $\pi'$. Of course, the same property is enjoyed by the symmetrical structures where $\pi(i)<\pi(i+1)< \pi(i+2)>\pi(i+3)>\pi(i+4)>\pi(i+5)$ and $\pi(i+1)<\pi(i+3)$ (and, again, with signs $<$ and $>$  swapped). So, we also call them \emph{lucky sixes}. See Figure~\ref{fig:lucky} for all 4 types of lucky sixes.

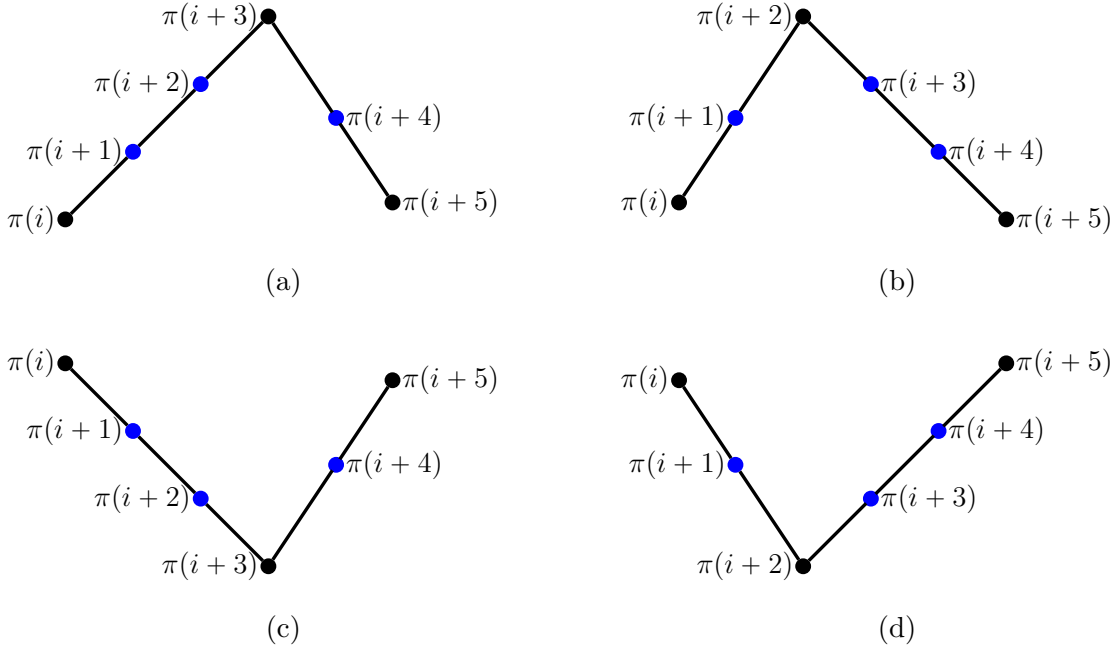
\begin{figure}[t]

\begin{subfigure}[b]{0.49\textwidth}
\scalebox{0.9}
{
\begin{tikzpicture}
[line width = .5pt,
vtx/.style={circle,draw,black,very thick,fill=black, line width = 1pt, inner sep=2pt},
vtx2/.style={circle,draw,blue,very thick,fill=black, line width = 1pt, inner sep=2pt},
]

    \node[vtx] (v1) at (0,0) {};
    \node[vtx2] (v2) at (1,1) {};
    \node[vtx2] (v3) at (2,2) {};
    \node[vtx] (v4) at (3,3) {};
    \node[vtx2] (v5) at (4,1.5) {};
    \node[vtx] (v6) at (4.8333,0.25) {};

    \fill[fill=black] (v1) circle (0.1) node [left] {$\pi(i)$};
    \draw[line width=0.5mm, color=black]  (v1) -- (v2);
    \fill[fill=blue] (v2) circle (0.1) node [left] {$\pi(i+1)$};
    \draw[line width=0.5mm, color=black]  (v2) -- (v3);
    \fill[fill=blue] (v3) circle (0.1) node [left] {$\pi(i+2)$};
    \draw[line width=0.5mm, color=black]  (v3) -- (v4);
    \fill[fill=black] (v4) circle (0.1) node [left] {$\pi(i+3)$};
    \draw[line width=0.5mm, color=black]  (v4) -- (v6);
    \fill[fill=blue] (v5) circle (0.1) node [right] {$\pi(i+4)$};
    \fill[fill=black] (v6) circle (0.1) node [right] {$\pi(i+5)$};
\end{tikzpicture}
}%scalebox
\caption[a]{}
\label{fig:lucky:1}
\end{subfigure}
\
\begin{subfigure}[b]{0.49\textwidth}
\scalebox{0.9}
{
\begin{tikzpicture}
[line width = .5pt,
vtx/.style={circle,draw,black,very thick,fill=black, line width = 1pt, inner sep=2pt},
vtx2/.style={circle,draw,blue,very thick,fill=black, line width = 1pt, inner sep=2pt},
]

    \node[vtx] (v1) at (0,0.25) {};
    \node[vtx2] (v2) at (0.8333,1.5) {};
    \node[vtx] (v3) at (1.8333,3) {};
    \node[vtx2] (v4) at (2.8333,2) {};
    \node[vtx2] (v5) at (3.8333,1) {};
    \node[vtx] (v6) at (4.8333,0) {};

    \fill[fill=black] (v1) circle (0.1) node [left] {$\pi(i)$};
    \draw[line width=0.5mm, color=black]  (v1) -- (v2);
    \fill[fill=blue] (v2) circle (0.1) node [left] {$\pi(i+1)$};
    \draw[line width=0.5mm, color=black]  (v2) -- (v3);
    \fill[fill=black] (v3) circle (0.1) node [left] {$\pi(i+2)$};
    \draw[line width=0.5mm, color=black]  (v3) -- (v6);
    \fill[fill=blue] (v4) circle (0.1) node [right] {$\pi(i+3)$};
    \fill[fill=blue] (v5) circle (0.1) node [right] {$\pi(i+4)$};
    \fill[fill=black] (v6) circle (0.1) node [right] {$\pi(i+5)$};

\end{tikzpicture}
}%scalebox
\caption[b]{}
\label{fig:lucky:2}
\end{subfigure}

\bigskip

\begin{subfigure}[b]{0.49\textwidth}
\scalebox{0.9}
{
\begin{tikzpicture}
[line width = .5pt,
vtx/.style={circle,draw,black,very thick,fill=black, line width = 1pt, inner sep=2pt},
vtx2/.style={circle,draw,blue,very thick,fill=black, line width = 1pt, inner sep=2pt},
]

    \node[vtx] (v1) at (0,3) {};
    \node[vtx2] (v2) at (1,2) {};
    \node[vtx2] (v3) at (2,1) {};
    \node[vtx] (v4) at (3,0) {};
    \node[vtx2] (v5) at (4,1.5) {};
    \node[vtx] (v6) at (4.8333,2.75) {};

    \fill[fill=black] (v1) circle (0.1) node [left] {$\pi(i)$};
    \draw[line width=0.5mm, color=black]  (v1) -- (v2);
    \fill[fill=blue] (v2) circle (0.1) node [left] {$\pi(i+1)$};
    \draw[line width=0.5mm, color=black]  (v2) -- (v3);
    \fill[fill=blue] (v3) circle (0.1) node [left] {$\pi(i+2)$};
    \draw[line width=0.5mm, color=black]  (v3) -- (v4);
    \fill[fill=black] (v4) circle (0.1) node [left] {$\pi(i+3)$};
    \draw[line width=0.5mm, color=black]  (v4) -- (v6);
    \fill[fill=blue] (v5) circle (0.1) node [right] {$\pi(i+4)$};
    \fill[fill=black] (v6) circle (0.1) node [right] {$\pi(i+5)$};

\end{tikzpicture}
}%scalebox
\caption[c]{}
\label{fig:lucky:3}
\end{subfigure}
\
\begin{subfigure}[b]{0.49\textwidth}
\scalebox{0.9}
{
\begin{tikzpicture}
[line width = .5pt,
vtx/.style={circle,draw,black,very thick,fill=black, line width = 1pt, inner sep=2pt},
vtx2/.style={circle,draw,blue,very thick,fill=black, line width = 1pt, inner sep=2pt},
]

    \node[vtx] (v1) at (0,2.75) {};
    \node[vtx2] (v2) at (0.8333,1.5) {};
    \node[vtx] (v3) at (1.8333,0) {};
    \node[vtx2] (v4) at (2.8333,1) {};
    \node[vtx2] (v5) at (3.8333,2) {};
    \node[vtx] (v6) at (4.8333,3) {};

    \fill[fill=black] (v1) circle (0.1) node [left] {$\pi(i)$};
    \draw[line width=0.5mm, color=black]  (v1) -- (v2);
    \fill[fill=blue] (v2) circle (0.1) node [left] {$\pi(i+1)$};
    \draw[line width=0.5mm, color=black]  (v2) -- (v3);
    \fill[fill=black] (v3) circle (0.1) node [left] {$\pi(i+2)$};
    \draw[line width=0.5mm, color=black]  (v3) -- (v6);
    \fill[fill=blue] (v4) circle (0.1) node [right] {$\pi(i+3)$};
    \fill[fill=blue] (v5) circle (0.1) node [right] {$\pi(i+4)$};
    \fill[fill=black] (v6) circle (0.1) node [right] {$\pi(i+5)$};

\end{tikzpicture}
}%scalebox
\caption[d]{}
\label{fig:lucky:4}
\end{subfigure}

\caption{Lucky sixes. The blue points appear in $\pi'$ as consecutive ones.}
\label{fig:lucky}
\end{figure}

Let $Y_1$ and $Y_2$  be the numbers of lucky sixes $\{i,i+1,i+2,i+3,i+4,i+5\}$ in $\Pi_n$ for, respectively, $1\le i\le n/2-3$ and $n/2-1\le i\le n-5$. Note that the probability that a given 6-tuple  is a lucky six is
$$4\times \frac{\binom 42}{6!}=\frac1{30}.$$
Indeed, considering, for instance, the number of ways to label by $1,\dots,6$, the lucky six in Figure~\ref{fig:lucky:1}, there is no question that 6 must be at the top, while 5 to its left. The remaining 4 values can be, however, distributed freely between the two pairs, $i,i+1$ and $i+4,i+5$. This explains $\binom 42$.
Thus,
$$\E(Y_1)=\E(Y_2)\sim\frac n{60}.$$
Again, a standard application of the Azuma inequality (Theorem~\ref{azuma}) yields that a.a.s. $Y_j=(1+o(1))\tfrac n{60}$, $j=1,2$.

Let $A_j$, $j=1,2$, be the alternating sub-sequences  in, respectively, $\{1,\dots, n/2\}$ and $\{n/2+1,\dots,n\}$, consisting of the extremal points of $\Pi_n$. Further, let $B_j$, $j=1,2$, be  alternating sub-sequences  in, respectively, $\{1,\dots, n/2\}$ and $\{n/2+1,\dots,n\}$,
consisting of he extremal points of $\Pi'_n$. By losing at most one point each, one can concatenate $A_j$ with $B_{3-j}$, $j=1,2$, obtaining the desired pair of alternating twins.
Noting that $|A_j\cup B_{3-j}|\sim \tfrac n3+\tfrac n{60}$ completes the proof of the lower bound in~\eqref{bounds}.
\end{proof}

\bigskip

\begin{proof}[Proof of Theorem~\ref{lem_ran_alt}, upper bound]
For the proof of the upper bound we need to consider two kinds of special 5-tuples.  A 5-tuple $\{i,i+1,i+2,i+3,i+4\}$ is called \emph{cornered} if either the first or the last four consecutive points form a monotone sub-sequence but all five do not (see Figure~\ref{fig:cornered}). A 5-tuple $\{i,i+1,i+2,i+3,i+4\}$ is called \emph{crooked} if the three middle points form a monotone sub-sequence but no four points do (see Figure~\ref{fig:crooked}).  Given a permutation $\pi$, let $e(\pi)$ be the number of extremal points in $\pi$, and let $co(\pi)$ and $cr(\pi)$ be, respectively, the number of cornered 5-tuples and the number of crooked 5-tuples in~$\pi$. The following crucial lemma sets an upper bound on the number of elements in two disjoint alternating sub-sequences of $\pi$ in terms of the three defined above parameters.

\begin{figure}[t]

\scalebox{0.9}
{

\begin{tikzpicture}
[line width = .5pt,
vtx/.style={circle,draw,black,very thick,fill=black, line width = 1pt, inner sep=2pt},
]

    \node[vtx] (v1) at (0,0) {};
    \node[vtx] (v2) at (1,1) {};
    \node[vtx] (v3) at (2,2) {};
    \node[vtx] (v4) at (3,3) {};
    \node[vtx] (v5) at (4,1.5) {};

    \fill[fill=black] (v1) circle (0.1) node [left] {$\pi(i)$};
    \draw[line width=0.5mm, color=black]  (v1) -- (v2);
    \fill[fill=black] (v2) circle (0.1) node [left] {$\pi(i+1)$};
    \draw[line width=0.5mm, color=black]  (v2) -- (v3);
    \fill[fill=black] (v3) circle (0.1) node [left] {$\pi(i+2)$};
    \draw[line width=0.5mm, color=black]  (v3) -- (v4);
    \fill[fill=black] (v4) circle (0.1) node [left] {$\pi(i+3)$};
    \draw[line width=0.5mm, color=black]  (v4) -- (v5);
    \fill[fill=black] (v5) circle (0.1) node [right] {$\pi(i+4)$};

\end{tikzpicture}
\qquad
\begin{tikzpicture}
[line width = .5pt,
vtx/.style={circle,draw,black,very thick,fill=black, line width = 1pt, inner sep=2pt},
]

    \node[vtx] (v1) at (0,1.5) {};
    \node[vtx] (v2) at (1,3) {};
    \node[vtx] (v3) at (2,2) {};
    \node[vtx] (v4) at (3,1) {};
    \node[vtx] (v5) at (4,0) {};

    \fill[fill=black] (v1) circle (0.1) node [left] {$\pi(i)$};
    \draw[line width=0.5mm, color=black]  (v1) -- (v2);
    \fill[fill=black] (v2) circle (0.1) node [right] {$\pi(i+1)$};
    \draw[line width=0.5mm, color=black]  (v2) -- (v3);
    \fill[fill=black] (v3) circle (0.1) node [right] {$\pi(i+2)$};
    \draw[line width=0.5mm, color=black]  (v3) -- (v4);
    \fill[fill=black] (v4) circle (0.1) node [right] {$\pi(i+3)$};
    \draw[line width=0.5mm, color=black]  (v4) -- (v5);
    \fill[fill=black] (v5) circle (0.1) node [right] {$\pi(i+4)$};

\end{tikzpicture}
}%scalebox

\medskip

\scalebox{0.9}
{
\begin{tikzpicture}
[line width = .5pt,
vtx/.style={circle,draw,black,very thick,fill=black, line width = 1pt, inner sep=2pt},
]

    \node[vtx] (v1) at (0,3) {};
    \node[vtx] (v2) at (1,2) {};
    \node[vtx] (v3) at (2,1) {};
    \node[vtx] (v4) at (3,0) {};
    \node[vtx] (v5) at (4,1.5) {};

    \fill[fill=black] (v1) circle (0.1) node [left] {$\pi(i)$};
    \draw[line width=0.5mm, color=black]  (v1) -- (v2);
    \fill[fill=black] (v2) circle (0.1) node [left] {$\pi(i+1)$};
    \draw[line width=0.5mm, color=black]  (v2) -- (v3);
    \fill[fill=black] (v3) circle (0.1) node [left] {$\pi(i+2)$};
    \draw[line width=0.5mm, color=black]  (v3) -- (v4);
    \fill[fill=black] (v4) circle (0.1) node [left] {$\pi(i+3)$};
    \draw[line width=0.5mm, color=black]  (v4) -- (v5);
    \fill[fill=black] (v5) circle (0.1) node [right] {$\pi(i+4)$};

\end{tikzpicture}
\qquad
\begin{tikzpicture}
[line width = .5pt,
vtx/.style={circle,draw,black,very thick,fill=black, line width = 1pt, inner sep=2pt},
]

    \node[vtx] (v1) at (0,1.5) {};
    \node[vtx] (v2) at (1,0) {};
    \node[vtx] (v3) at (2,1) {};
    \node[vtx] (v4) at (3,2) {};
    \node[vtx] (v5) at (4,3) {};

    \fill[fill=black] (v1) circle (0.1) node [left] {$\pi(i)$};
    \draw[line width=0.5mm, color=black]  (v1) -- (v2);
    \fill[fill=black] (v2) circle (0.1) node [right] {$\pi(i+1)$};
    \draw[line width=0.5mm, color=black]  (v2) -- (v3);
    \fill[fill=black] (v3) circle (0.1) node [right] {$\pi(i+2)$};
    \draw[line width=0.5mm, color=black]  (v3) -- (v4);
    \fill[fill=black] (v4) circle (0.1) node [right] {$\pi(i+3)$};
    \draw[line width=0.5mm, color=black]  (v4) -- (v5);
    \fill[fill=black] (v5) circle (0.1) node [right] {$\pi(i+4)$};

\end{tikzpicture}

}%scalebox

\caption{Cornered 5-tuples.}
\label{fig:cornered}
\end{figure}
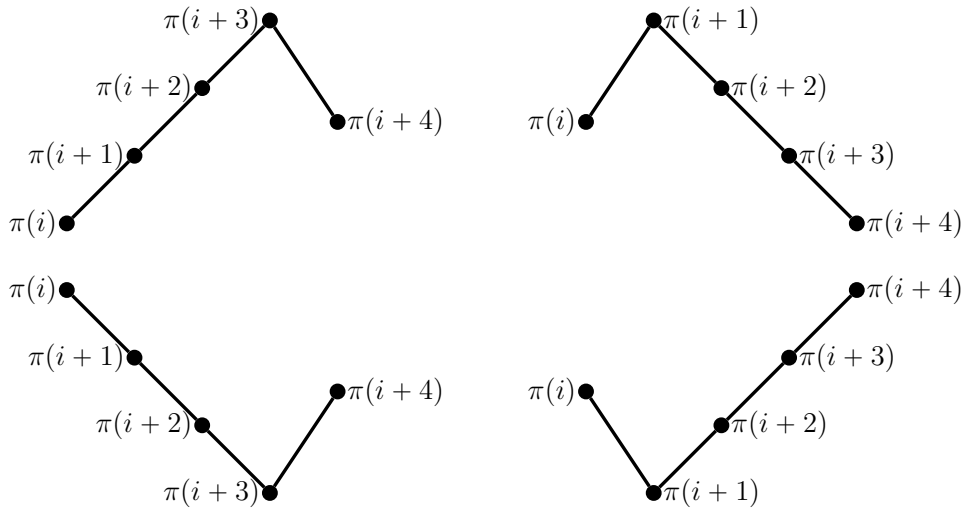

\begin{lemma}\label{crucial}
	  Let $A$ and $B$ be two disjoint alternating sub-sequences in a permutation $\pi$ of $[n]$. Then
\begin{equation}\label{bound}
|A|+|B|\le e(\pi)+co(\pi)+cr(\pi).
\end{equation}
\end{lemma}

Deferring the proof of Lemma \ref{crucial} for later, we now deduce from it the  upper bound in~\eqref{bounds}. Let $L$ count the cornered 5-tuples in the random permutation $\Pi_n$ and let $Z$ count the crooked 5-tuples in $\Pi_n$. Note that the probability that a given 5-tuple  is cornered is
$4\times\binom43/5!=\tfrac8{60}$ and so, $\E(L)=\tfrac8{60}\times(n-4)$. Note also the that the probability that a given 5-tuple is crooked is $2\times\tfrac{11}{5!}=\tfrac{11}{60}$ (see Figure~\ref{fig:crooked_prob}) and so, $\E(W)=\tfrac{11}{60}\times(n-4)$. Another application of the Azuma inequality (Theorem~\ref{azuma})  yields that a.a.s. $L=(1+o(1))\tfrac {8n}{60}$, while $Z=(1+o(1))\tfrac{11n}{60}$. Plugging into \eqref{bound}, we finally obtain that
$$\alpha_n\le\frac12(1+o(1))\left(\frac23+\frac8{60}+\frac{11}{60}\right)n=\left(\frac12-\frac{1}{120}+o(1)\right)n. $$
\end{proof}

\begin{figure}

\scalebox{0.9}
{

\begin{tikzpicture}
[line width = .5pt,
vtx/.style={circle,draw,black,very thick,fill=black, line width = 1pt, inner sep=2pt},
]
    \node[vtx] (v1) at (0,1.5) {};
    \node[vtx] (v2) at (1,0) {};
    \node[vtx] (v3) at (2,1) {};
    \node[vtx] (v4) at (3,2) {};
    \node[vtx] (v5) at (4,0.5) {};

    \fill[fill=black] (v1) circle (0.1) node [left] {$\pi(i)$};
    \draw[line width=0.5mm, color=black]  (v1) -- (v2);
    \fill[fill=black] (v2) circle (0.1) node [right] {$\pi(i+1)$};
    \draw[line width=0.5mm, color=black]  (v2) -- (v3);
    \fill[fill=black] (v3) circle (0.1) node [right] {$\pi(i+2)$};
    \draw[line width=0.5mm, color=black]  (v3) -- (v4);
    \fill[fill=black] (v4) circle (0.1) node [right] {$\pi(i+3)$};
    \draw[line width=0.5mm, color=black]  (v4) -- (v5);
    \fill[fill=black] (v5) circle (0.1) node [right] {$\pi(i+4)$};

\end{tikzpicture}
\qquad
\begin{tikzpicture}
[line width = .5pt,
vtx/.style={circle,draw,black,very thick,fill=black, line width = 1pt, inner sep=2pt},
]

    \node[vtx] (v1) at (0,0.5) {};
    \node[vtx] (v2) at (1,2) {};
    \node[vtx] (v3) at (2,1) {};
    \node[vtx] (v4) at (3,0) {};
    \node[vtx] (v5) at (4,1.5) {};

    \fill[fill=black] (v1) circle (0.1) node [left] {$\pi(i)$};
    \draw[line width=0.5mm, color=black]  (v1) -- (v2);
    \fill[fill=black] (v2) circle (0.1) node [left] {$\pi(i+1)$};
    \draw[line width=0.5mm, color=black]  (v2) -- (v3);
    \fill[fill=black] (v3) circle (0.1) node [left] {$\pi(i+2)$};
    \draw[line width=0.5mm, color=black]  (v3) -- (v4);
    \fill[fill=black] (v4) circle (0.1) node [left] {$\pi(i+3)$};
    \draw[line width=0.5mm, color=black]  (v4) -- (v5);
    \fill[fill=black] (v5) circle (0.1) node [right] {$\pi(i+4)$};

\end{tikzpicture}
}%scalebox

\caption{Crooked 5-tuples.}
\label{fig:crooked}
\end{figure}
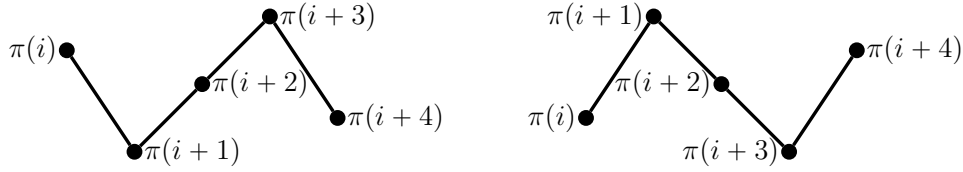

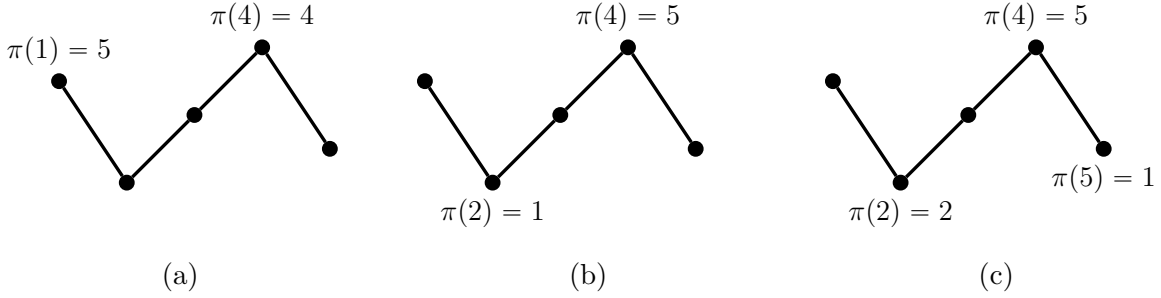
\begin{figure}
\begin{center}
     \begin{subfigure}[b]{0.32\textwidth}
\scalebox{0.9}
{
\centering
\begin{tikzpicture}
[line width = .5pt,
vtx/.style={circle,draw,black,very thick,fill=black, line width = 1pt, inner sep=2pt},
]

    \node[vtx] (v1) at (0,1.5) {};
    \node[vtx] (v2) at (1,0) {};
    \node[vtx] (v3) at (2,1) {};
    \node[vtx] (v4) at (3,2) {};
    \node[vtx] (v5) at (4,0.5) {};

    \fill[fill=black, outer sep=1mm]  (v1) circle (0.1) node [above] {$\pi(1)=5$};
    \draw[line width=0.5mm, color=black]   (v1) -- (v2);
    \fill[fill=black, outer sep=1mm] (v2) circle (0.1) node [below] {\phantom{$\pi(2)=1$}};
    \draw[line width=0.5mm, color=black]  (v2) -- (v3);
    \fill[fill=black] (v3) circle (0.1) node [right] {};
    \draw[line width=0.5mm, color=black]  (v3) -- (v4);
    \fill[fill=black, outer sep=1mm] (v4) circle (0.1) node [above] {$\pi(4)=4$};
    \draw[line width=0.5mm, color=black]  (v4) -- (v5);
    \fill[fill=black] (v5) circle (0.1) node [right] {};
\end{tikzpicture}
}
\caption[a]{}
\label{fig:crooked_prob:1}
\end{subfigure}
     \
\begin{subfigure}[b]{0.32\textwidth}
\scalebox{0.9}
{
\centering
\begin{tikzpicture}
[line width = .5pt,
vtx/.style={circle,draw,black,very thick,fill=black, line width = 1pt, inner sep=2pt},
]

    \node[vtx] (v1) at (0,1.5) {};
    \node[vtx] (v2) at (1,0) {};
    \node[vtx] (v3) at (2,1) {};
    \node[vtx] (v4) at (3,2) {};
    \node[vtx] (v5) at (4,0.5) {};

    \fill[fill=black] (v1) circle (0.1) node [left] {};
    \draw[line width=0.5mm, color=black]  (v1) -- (v2);
    \fill[fill=black, outer sep=1mm] (v2) circle (0.1) node [below] {$\pi(2)=1$};
    \draw[line width=0.5mm, color=black]  (v2) -- (v3);
    \fill[fill=black] (v3) circle (0.1) node [right] {};
    \draw[line width=0.5mm, color=black]  (v3) -- (v4);
    \fill[fill=black, outer sep=1mm] (v4) circle (0.1) node [above] {$\pi(4)=5$};
    \draw[line width=0.5mm, color=black]  (v4) -- (v5);
    \fill[fill=black] (v5) circle (0.1) node [below] {};
\end{tikzpicture}
}
\caption[b]{}
\label{fig:crooked_prob:2}
\end{subfigure}
     \
\begin{subfigure}[b]{0.32\textwidth}
\scalebox{0.9}
{
\centering
\begin{tikzpicture}
[line width = .5pt,
vtx/.style={circle,draw,black,very thick,fill=black, line width = 1pt, inner sep=2pt},
]

    \node[vtx] (v1) at (0,1.5) {};
    \node[vtx] (v2) at (1,0) {};
    \node[vtx] (v3) at (2,1) {};
    \node[vtx] (v4) at (3,2) {};
    \node[vtx] (v5) at (4,0.5) {};

    \fill[fill=black] (v1) circle (0.1) node [left] {};
    \draw[line width=0.5mm, color=black]  (v1) -- (v2);
    \fill[fill=black, outer sep=1mm] (v2) circle (0.1) node [below] {$\pi(2)=2$};
    \draw[line width=0.5mm, color=black]  (v2) -- (v3);
    \fill[fill=black] (v3) circle (0.1) node [right] {};
    \draw[line width=0.5mm, color=black]  (v3) -- (v4);
    \fill[fill=black, outer sep=1mm] (v4) circle (0.1) node [above] {$\pi(4)=5$};
    \draw[line width=0.5mm, color=black]  (v4) -- (v5);
    \fill[fill=black, outer sep=1mm] (v5) circle (0.1) node [below] {$\pi(5)=1$};
\end{tikzpicture}
}
\caption[c]{}
\label{fig:crooked_prob:3}
\end{subfigure}
\end{center}

\caption{ \subref{fig:crooked_prob:1}~If $\pi(1)=5$, then $\pi(4)=4$ and the remaining number of choices is $\binom{3}{2}$. \subref{fig:crooked_prob:2}~If $\pi(4)=5$ and $\pi(2)=1$, then we have $3!$ choices. \subref{fig:crooked_prob:3}~Finally, if $\pi(4)=5$ and $\pi(5)=1$, then there are $2!$ remaining choices.}
\label{fig:crooked_prob}
\end{figure}

It remains to prove  Lemma \ref{crucial}.

\begin{proof}[Proof of Lemma~\ref{crucial}] Let $A$ and $B$ be given as in the lemma. Let $E$ be the set of extremal points in $\pi$ and $F$ -- the set of points neighboring the extremal points but not extremal themselves.
We are going to construct an injective mapping $\phi: A\cup B\to E\cup F$. Then, noting that $|F|=co(\pi)+cr(\pi)$, completes the proof.

Let $j_1,\dots,j_k$ be all extremal points in $\pi$. These points divide the whole range $[n]$ into monotone segments
$$\pi_i=(\pi(j_i),\pi(j_i+1),\dots, \pi(j_i+\ell_i),\pi(j_{i+1}))$$
$i=1,\dots,k-1$. Note that the number of inner points of $\pi_i$, $\ell_i$ can equal 0. Before constructing the  desired mapping $\phi$, let us examine the distribution of the set $A\cup B$ among the segments $\pi_i$. Our first observation is that each segment contains at most two elements of $A$ and at most two elements of $B$. Moreover, if $\pi_i$ contains exactly two elements of $A$, then one of them is a minimal element of $A$ and the other -- a maximal element of $A$, and the same is true for $B$. But most crucial is the following property concerning a pair of consecutive segments $\pi_i$ and $\pi_{i+1}$. If $j_{i+1}$ is maximal, respectively, minimal in $\pi$, then there is in total at most one maximal, resp., minimal element of $A$ on these segments.

Knowing all this, it is easy to see that the following construction is, indeed, an injection. If $\pi_i$ is increasing, then to the maximal elements of $A\cup B$ lying on $\pi_i$, assign the top-most two elements of $\pi_i$, that is, to $\pi(j_i+\ell_i),\pi(j_{i+1})$, in any feasible fashion. While to the minimal elements of $A\cup B$ lying on $\pi_i$ assign the two down-most elements of $\pi_i$, that is, to $\pi(j_i),\pi(j_i+1)$. If $\pi_i$ is decreasing, we proceed similarly, but with the pairs $\pi(j_i+\ell_i),\pi(j_{i+1})$ and $\pi(j_i),\pi(j_i+1)$ swapped.
\end{proof}

\section{Proof of Proposition \ref{alt_the_king}}\label{proof_prop}

Given a sequence $s=(s_1,\dots,s_r)$ with $s_i\in\{+,-\}$ and a linearly ordered set $S$ of size $|S|=r+1$, denote by $\cN_S(s)$  the set of all permutations $\pi$ of $S$ with the shape $s(\pi)=s$. If $S=[r+1]$, then we abbreviate $\cN(s):=\cN_{[r+1]}(s)$. Further, let $N_S(s)=|\cN_S(s)|$. Observe that $N_S(s)$ does not depend on $S$, so we skip the subscript $_S$ altogether here.

The complement of a sequence $s=(s_1,\dots,s_r)$ is naturally defined as the sequence $\bar s=(\bar s_1,\dots,\bar s_r)$, where $\{s_i,\bar s_i\}=\{+,-\}$ for each $i$. In other words, one replaces each $+$ in $s$ with $-$, and vice versa. It is easy to see that $N(s)=N(\bar s)$.

Recall that $A_n=N(a^{(n)}_+)=N(a^{(n)}_-)$. Our proof of Proposition \ref{alt_the_king} is by induction on $n$ and, in its final accord, utilizes the following known identity involving the sequence $A_n$ (see, e.g.,~\cite{StanleySurvey}):
\begin{equation}\label{Stan}
\sum_{k=0}^n\binom nkA_kA_{n-k}=2A_{n+1}.
\end{equation}

What is more, our proof is also inspired by the  idea behind the proof of \eqref{Stan}, which is to build a permutation of $[n+1]$ beginning with positioning the element $n+1$, and then separately counting the completions to the left and to the right of it.
Also, as the R-H-S of \eqref{Stan} is a double of what we want, we are doomed to count in permutations with the complementary shape as well.

\begin{proof}[Proof of Proposition \ref{alt_the_king}] For $n\le3$, the proposition follows by inspection. Fix $n\ge3$ and assume it is true for all $n'\le n$. Given a shape $s^{(n)}:=s=(s_1,\dots,s_n)$ our goal is to show that $N(s)\le A_{n+1}$.

For each $k=0,1,\dots,n$, let $\cN_k(s)=\{\pi\in\cN(s): \pi_{k+1}=n+1\}$. As $n+1$ is always a maximum element of $\pi$, we have $\cN_k(s)\neq\emptyset$ if and only if $s_k=+$ and $s_{k+1}=-$. Thus, setting $K^{\wedge}=\{k: s_k=+ \text{ and } s_{k+1}=-\}$, we have $\cN(s)=\bigcup_{k\in K^\wedge}\cN_k(s)$, and, as the sets under the union are obviously disjoint, $N(s)=\sum_{k\in K^\wedge}N_k(s)$, where $N_k(s)=|\cN_k(s)|$.
For a fixed $k$, let us focus  on the number $N_k(s)$. Every permutation in $\cN_k(s)$ consists of a `prefix' $u$, followed by $n+1$, followed by a suffix $v$. Introducing `truncated' shapes $s_k'=(s_1,\dots,s_{k-1})$ and $s_k''=(s_{k+2},\dots,s_n)$, $u$ and $v$ must satisfy $s(u)=s_k'$ and $s(v)=s_k''$. Hence, using also the induction assumption,
$$N_k(s)=\binom nk N(s_k')N(s_k'')\le \binom nk A_kA_{n-k}.$$

The same is true for the complementary shape $\bar s$ as well. Recalling that $N(\bar s)=N(s)$ and noticing that the set
$$\{k: \bar s_k=+ \text{ and } \bar s_{k+1}=-\}=\{k: s_k=- \text{ and } s_{k+1}=+\}=:K^\vee$$
is disjoint from $K^\wedge$, we thus conclude that
$$2N(s)\le\sum_{k\in K^{\wedge}\cup K^\vee}\binom nk A_kA_{n-k}\le\sum_{k=0}^n\binom nkA_kA_{n-k}=2A_{n+1},$$
where the last equality is \eqref{Stan}.

\end{proof}
\section{Concluding remarks}

We believe that the lower bound in Theorem~\ref{thm_weak} can be improved and it is plausible to conjecture that $wt(n) \sim \frac{n}{2}$. As a matter of fact, if $\pi$ happens to be an $n$-permutation with $e(\pi) = o(n)$, then the construction used in the proof of~\eqref{bounds_gen} yields~$wt(\pi)\sim\frac{n}{2}$.

It is also not difficult to see that the lower bound on~$\alpha_n$ in Theorem~\ref{lem_ran_alt} can be improved. Let $\pi'$ be the sub-permutation obtained from $\pi$ by removing all the extremal points of~$\pi$. Recall that in the proof of the lower bound~\eqref{bounds} we estimated $e(\pi')$ by using the lucky six tuples. But one can also consider more ``lucky'' structures. This can be done by incorporating zigzags into the lucky six tuples as in Figure~\ref{fig:newlucky}, for example. This already gives an improvement on the lower bound on $\alpha_n$:
\[
\alpha_n \ge \left(\frac{1}{3} + \frac{1}{60} + \frac{1}{2} \cdot 4 \cdot \frac{117+105}{8!}+o(1)\right)n.
\]
Now we can consider longer lucky tuples (of length $10, 12, 14,...$) and use computer to calculate the corresponding expectations. Computer simulations suggest that
\[
\alpha_n\ge (1/3 + 0.1006...)n.
\]
We do not know what the exact value of the second term is here, since it is not even clear how to compute the expected value $\E(e(\pi'))$.

Finally, let us define, for even $n$, the function $T(n)$ which counts all $n$-permutations that are weak twins of length $n/2$, that is, all $n$-permutations that can be split into two sub-permutations with the same pattern. What is the asymptotic growth of $T(n)$?
\begin{figure}[t]

\begin{subfigure}[b]{0.3\textwidth}
\scalebox{0.9}
{
\begin{tikzpicture}
[line width = .5pt,
vtx/.style={circle,draw,black,very thick,fill=black, line width = 1pt, inner sep=2pt},
vtx2/.style={circle,draw,blue,very thick,fill=blue, line width = 1pt, inner sep=2pt},
]

    \node[vtx] (v1) at (0,0) {};
    \node[vtx2] (v2) at (1,1) {};
    \node[vtx2] (v3) at (2,2) {};
    \node[vtx] (v4) at (3,3) {};
    \node[vtx] (v4a) at (3.5,2.25) {};
    \node[vtx] (v4b) at (3.9,3.5) {};

    \node[vtx2] (v5) at (4,1.5) {};
    \node[vtx] (v6) at (4.8333,0.25) {};

    \draw[line width=0.5mm, color=black]  (v1) -- (v2);
    \draw[line width=0.5mm, color=black]  (v2) -- (v3);
    \draw[line width=0.5mm, color=black]  (v3) -- (v4);
    \draw[line width=0.5mm, color=black]  (v4) -- (v4a);
    \draw[line width=0.5mm, color=red]  (v4a) -- (v4b);
    \draw[line width=0.5mm, color=red]  (v4b) -- (v5);
    \draw[line width=0.5mm, color=black]  (v5) -- (v6);

    \end{tikzpicture}}%scalebox
\caption[b]{}
\label{fig:newlucky:1}
\label{fig:}
\end{subfigure}
\qquad\qquad
\begin{subfigure}[b]{0.3\textwidth}
\scalebox{0.9}
{
\begin{tikzpicture}
[line width = .5pt,
vtx/.style={circle,draw,black,very thick,fill=black, line width = 1pt, inner sep=2pt},
vtx2/.style={circle,draw,blue,very thick,fill=blue, line width = 1pt, inner sep=2pt},
]

    \node[vtx] (v1) at (0,0) {};
    \node[vtx2] (v2) at (1,1) {};
    \node[vtx] (v2a) at (1.5,1.5) {};
    \node[vtx] (v2b) at (1.9,0.5) {};
    \node[vtx2] (v3) at (2,2) {};
    \node[vtx] (v4) at (3,3) {};
    \node[vtx2] (v5) at (4,1.5) {};
    \node[vtx] (v6) at (4.8333,0.25) {};

    \draw[line width=0.5mm, color=black]  (v1) -- (v2);
    \draw[line width=0.5mm, color=black]  (v2) -- (v2a);
    \draw[line width=0.5mm, color=red]  (v2a) -- (v2b);
    \draw[line width=0.5mm, color=red]  (v2b) -- (v3);
    \draw[line width=0.5mm, color=black]  (v3) -- (v4);
    \draw[line width=0.5mm, color=black]  (v4) -- (v6);
    \node[vtx2] (v5) at (4,1.5) {};

\end{tikzpicture}
}%scalebox
\caption[b]{}
\label{fig:newlucky:2}
\end{subfigure}

\caption{\subref{fig:newlucky:1} $117$ choices; \subref{fig:newlucky:2} 105 choices.}
\label{fig:newlucky}
\end{figure}
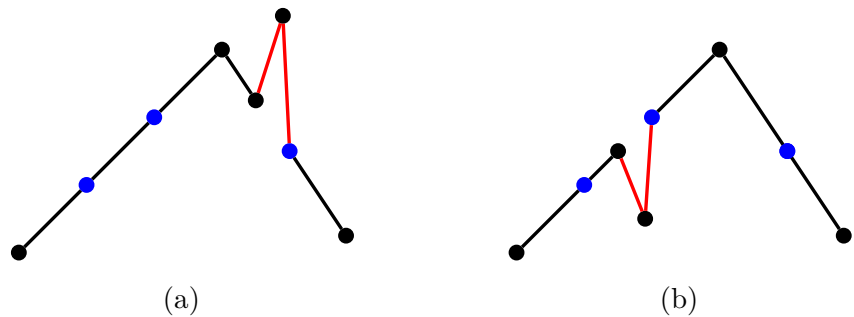

\end{document}